\documentclass[12pt,a4paper]{article}

\usepackage{amsmath,amssymb,latexsym,amsthm}
\usepackage[english]{babel}
\usepackage[latin2]{inputenc}

\newtheorem{thm}{Theorem}
 
 \newtheorem{Lemma}[thm]{Lemma}
 
 \theoremstyle{definition}
 
 \theoremstyle{remark}
 \newtheorem{Rem}[thm]{Remark}

\begin{document}

\title{The multiplication groups of 2-dimensional \\ topological loops}
\author{\'Agota Figula} 
\date{}
\maketitle

\begin{abstract} 
We prove that if the multiplication group $Mult(L)$ of a connected $2$-dimensional  topological loop is a 
Lie group, then $Mult(L)$ is an elementary filiform nilpotent Lie group of dimension at least $4$. Moreover, we describe loops having elementary filiform Lie groups $\mathbb F$ as the group topologically generated by their left translations and obtain a complete classification for these loops $L$ if $\hbox{dim} \ \mathbb F=3$. In this case necessary and sufficient conditions for $L$ 
are given that $Mult(L)$ is an elementary filiform Lie group for a given allowed dimension. 
\end{abstract} 



\bigskip
\centerline{\bf Introduction} 

\bigskip
\noindent
The multiplication group $Mult(L)$ and the inner mapping group $Inn(L)$ of a loop $L$ 
introduced in \cite{albert}, \cite{bruck}, are important tools since they reflect strongly the structure of $L$. In particular, there is a strong correspondence between the normal subloops of $L$ and certain normal subgroups of $Mult(L)$.   
Hence it is an interesting question which groups can be represented as multiplication groups $Mult(L)$ of loops (cf. \cite{niemenmaa}, \cite{vesanen}, \cite{vesanen2}). A purely group theoretical characterization of multiplication groups is given in \cite{kepka}. 

Topological and differentiable loops such that the groups $G$ topologically generated by the left translations are Lie groups have been studied in \cite{loops}. There the topological loops $L$ are treated as continuous sharply transitive sections $\sigma :G/H \to G$, where $H$ is the stabilizer of the identity of $L$ in $G$. In \cite{loops} and \cite{figula} it is proved that essentially up to two exceptions any connected $3$-dimensional Lie group occurs as the group topologically generated by the left translations of a connected topological $2$-dimensional loop. These 
exceptions are groups which are either locally isomorphic to the connected component of the group of motions or isomorphic to the connected component of the group of dilatations of the euclidean plane. In contrast to these results if the group $Mult(L)$ topologically generated by  all left and right translations of a connected $2$-dimensional topological loop $L$ is a Lie group, then the isomorphism type of 
$Mult(L)$ as well as of $L$ are strongly restricted. This shows our theorems in which Lie groups the Lie algebras of which are filiform (cf. \cite{goze}, pp. 615-663) play a fundamental role. We call a simply connected Lie group of dimension $n+2$ an elementary filiform Lie group 
$\mathbb F_{n+2}$ if its Lie algebra has a basis $\{ e_1, \cdots , e_{n+2} \}$, $n \ge 1$ such that  $[e_1, e_i]= (n+2-i) e_{i+1}$ for $2 \le i \le n+1$ and all other products are zero. With this notion we can formulate our theorems as follows: 

\begin{thm} \label{nilpotent} Let $L$ be a connected $2$-dimensional topological proper loop. The group $Mult(L)$ 
topologically generated by all translations of $L$ is a Lie group if and only if $Mult(L)$ is an elementary filiform Lie group $\mathbb F_{n+2}$ with $n \ge 2$ and the group $G$ topologically generated by the left translations of $L$ is an elementary filiform Lie group $\mathbb F_{m+2}$, where $1 \le m \le n$. Moreover, the inner mapping group $Inn(L)$ corresponds to the abelian subalgebra $\langle e_2, e_3, \cdots ,e_{n+1} \rangle $. \end{thm} 

\medskip 
\noindent
\begin{thm} Let $G$ be the elementary filiform Lie group $\mathbb F_{n+2}$ with $n \ge 1$. Then $G$ is isomorphic to the group topologically generated by the left translations of a connected  simply connected 
 $2$-dimensional topological proper loop $L$ the multiplication of which is given by 
{\small \begin{equation} (u_1,z_1) \ast (u_2,z_2)=(u_1+u_2, z_1+z_2-u_2 v_1(u_1)+u_2^2 v_2(u_1)+ \cdots +(-1)^n u_2^n v_n(u_1)), \end{equation} }
\noindent
where for every $i \in \{1, 2, \cdots ,n \}$ the non-constant functions $v_i: \mathbb R \to \mathbb R$ are continuous with $v_i(0)=0$ such that the function $v_n$ is non-linear. 

The group $G$ coincides with the group $Mult(L)$ topologically generated by all left and right  translations 
of $L$ if and only if there are continuous functions $s_i: \mathbb R \to \mathbb R$, $i=1, \cdots ,n$  depend on the variable $u$ such that the equation 
{\small \begin{equation} -x s_1(u)+ x^2 s_2(u)+ \cdots +(-1)^n x^n s_n(u)-v_1(u) x +v_2(u) x^2+ \cdots +(-1)^n x^n v_n(u) \nonumber \end{equation}
\begin{equation} =-u v_1(x)+ u^2 v_2(x)-u^3 v_3(x)+ \cdots +(-1)^n u^n v_n(x) \nonumber \end{equation} }
\noindent
holds. 

Moreover, $L$ is commutative if the continuous functions 
$v_i: \mathbb R \to \mathbb R$, $i=1,2, \cdots ,n$ satisfy the matrix equation
\[ \left ( \begin{array}{c} 
v_1(x) \\
v_2(x) \\
\vdots  \\
v_n(x) \end{array} \right ) = A \left ( \begin{array}{c} 
x \\
x^2 \\
\vdots \\ 
x^n \end{array} \right ), \]
where $A \in M_n(\mathbb R)$ such that $a_{ij}=(-1)^{i+j} a_{ji}$ for all $i,j \in \{ 1,2, \cdots ,n \}$. \end{thm} 

\medskip
\noindent
\begin{thm} If $L$ is a $2$-dimensional connected simply connected  topological loop having 
the filiform Lie group $\mathbb F_3$ as the group topologically generated by the left translations 
of $L$ then the multiplication of $L$ is given by 
\noindent
\begin{equation} (u_1,z_1) \ast (u_2,z_2)=(u_1+u_2, z_1+z_2-u_2 v_1(u_1)), \end{equation} 
where $v_1: \mathbb R \to \mathbb R$ is a non-linear continuous function with $v_1(0)=0$. 

The group $Mult(L)$ topologically generated by all left and right translations of $L$ is isomorphic to the filiform Lie group 
$\mathbb F_{n+2}$ for $n <1$ if and only if the continuous function $v_1:\mathbb R \to \mathbb R$ has the shape 
\noindent
\begin{equation} 
v_1(u)=(-1)^{n+1} a_n u^n+(-1)^{n} a_{n-1} u^{n-1}+ \cdots +a_1 u,  \nonumber \end{equation}  
\noindent
where $a_1, \cdots , a_n$ are real numbers such that  $a_n$ is different from $0$.
\end{thm} 

If the loop $L$ of Theorem \ref{nilpotent} is simply connected then it is a central extension 
of the group $\mathbb R$ by the group $\mathbb R$ (cf. Theorem 28.1 in \cite{loops}, p. 338).  Hence it is a centrally nilpotent loop of class $2$ and can be represented by a multiplication 
over $\mathbb R \times \mathbb R$. If $L$ is not simply connected then $L$ is homeomorphic to the cylinder $\mathbb R \times \mathbb R/ \mathbb Z$. The centre of the multiplication group of 
$L$ is isomorphic to the group $SO_2(\mathbb R)$ (cf. Theorem 28.1 in \cite{loops}, p. 338).

\bigskip
\noindent
\centerline{\bf Preliminaries} 

\bigskip
\noindent
A binary system $(L, \cdot )$ is called a loop if there exists an element 
$e \in L$ such that $x=e \cdot x=x \cdot e$ holds for all $x \in L$ and the 
equations $a \cdot y=b$ and $x \cdot a=b$ have precisely one solution which we denote 
by $y=a \backslash b$ and $x=b/a$. 
\newline
The left and right translations $\lambda _a= y \mapsto a \cdot y :L \to L$ and 
$\rho _a: y \mapsto y \cdot a:L \to L$, $a \in L$  are permutations of $L$. 
\newline 
The permutation group $Mult(L)=\langle L_a, R_a; \ a \in L \rangle $ is called the multiplication group of $L$. 
The stabilizer of the identity element $e \in L$ in $Mult(L)$ is denoted by $Inn(L)$ and  $Inn(L)$ is called  the inner mapping group of $L$. 

Let $G$ be a group, $H \le G$, and let $A$ and $B$ be two left transversals to $H$ in $G$ (i.e. two systems of representatives for the left cosets of $H$ in $G$). We say that the two left transversals $A$ and $B$ are $H$-connected if $a^{-1} b^{-1} a b \in H$ for every $a \in A$ and $b \in B$. By $C_G(H)$ we denote the core of $H$ in $G$ (the largest normal subgroup of $G$ contained in $H$). If $L$ is a loop then $\Lambda (L)=\{ \lambda _a; \ a \in L \}$ and $\mathcal{R} (L)=\{ \rho _a; \ a \in L \}$ are $Inn(L)$-connected transversals in the group $Mult(L)$ and the core of $Inn(L)$ in $Mult(L)$ is trivial. The connection between multiplication groups of loops and transversals is given in \cite{kepka} by Theorem 4.1.   
This theorem yields the following 

\begin{Lemma} \label{kepka} Let $L$ be a loop and $\Lambda (L)$ be the set of left translations 
of $L$. Let $G$ be a group containing $\Lambda (L)$ and a subgroup $H$ with $C_G(H)=1$ 
such that $\Lambda (L)$ is a left transversal to $H$ in $G$. The group $G$ is isomorphic to the 
multiplication group 
$Mult(L)$ of $L$ if and only if  there is a left transversal $T$ to $H$ in $G$ such that $\Lambda (L)$ and $T$ are $H$-connected and $G=\langle \Lambda (L), T \rangle $. In this case $H$ is isomorphic to the inner mapping group $Inn(L)$ of $L$.  
\end{Lemma}

The kernel of a homomorphism $\alpha :(L, \circ ) \to (L', \ast )$ 
of a loop $L$ into a loop $L'$ is a normal subloop $N$ of $L$. In our paper we need the following facts (cf. \cite{bruck2}, Lemma 1.3, p. 62).

\begin{Lemma} \label{bruck} Let $L$ be a loop with multiplication group $Mult(L)$, with inner mapping group $Inn(L)$ and with identity element $e$. 

\medskip
\noindent
(i) Let $\alpha $ be a homomorphism, with kernel $N$ of the loop $L$ onto the loop $\alpha (L)$. Then $\alpha $ induces a homomorphism of the group $Mult(L)$ onto the group $Mult(\alpha (L))$. 

Denote by $M(N)$ the set $\{ m \in Mult(L); \ x N=m(x) N,\ \hbox{for \ all} \ x \in L \}$. 
Then $M(N)$ is a normal subgroup of $Mult(L)$ and the multiplication group of the factor loop $L/N$ is isomorphic to $Mult(L)/M(N)$.

\medskip
\noindent
(ii) For every normal subgroup $\mathcal{N}$ of $Mult(L)$ the orbit $\mathcal{N}(e)$ is a normal subloop of $L$. Moreover, $\mathcal{N} \le M(\mathcal{N}(e))$.  
\end{Lemma}

The theory of topological loops $L$ is the theory of the 
continuous binary operations  
$(x,y) \mapsto x \cdot y, \ (x,y) \mapsto x/y, \ (x,y) \mapsto x 
\backslash y $ on the topological manifold $L$. If $L$ is a connected topological loop 
then the left translations $\lambda _a$ as well as the right translations $\rho _a$, $a \in L$  are homeomorphisms of $L$.

Every connected topological loop having a Lie group  as the group topologically generated by the left translations is realized on a homogeneous space $G/H$, where $G$ is a connected Lie group, $H$ is a closed subgroup containing no non-trivial normal subgroup of $G$ and $\sigma :G/H \to G$ is  a continuous sharply transitive section with 
$\sigma (H)=1 \in G$ such that the subset $\sigma (G/H)$ generates $G$.  The multiplication  of  $L$ on the  
space $G/H$  is  defined by $x H \ast y H=\sigma (x H) y H$ and  the group $G$ is  the group topologically generated by the left translations of $L$. Moreover, the subgroup $H$ is the stabilizer of the identity element 
$e \in L$ in the group $G$ (cf. \cite{loops}, Section 1.3).

\bigskip
\noindent
\centerline{\bf Proofs }  

\bigskip
\noindent
If $L$ is a connected topological loop having a Lie group as the group $Mult(L)$ topologically generated by all left and right translations, then $Mult(L)$ acts transitively and effectively as a topological transformation group on $L$. All transitive transformation groups on a $2$-dimensional manifold are classified by Lie \cite{lie} and Mostow in \cite{mostow}, §10, pp. 625-635.

\begin{Lemma} \label{affinities} 
If the group $Mult(L)$ topologically generated by all left and right translations of a $2$-dimensional connected topological proper loop is a Lie group then the group $Mult(L)$ is solvable.   
\end{Lemma}
\begin{proof} Let $L$ be a $2$-dimensional  connected topological proper loop such that the group $Mult(L)$ topologically generated by all left and right translations is a non-solvable Lie group.   
We may assume that $L$ is simply connected, otherwise we would consider the universal covering of $L$. Then $L$ is homeomorphic to $\mathbb R^2$ (cf. Theorem 19.1 in \cite{loops}, p. 249). 

If the radical $\mathcal R$ of $Mult(L)$ is trivial then the Lie algebra of the group $Mult(L)$ is semisimple and according to \cite{mostow} (§10) the Lie group $Mult(L)$  
is locally isomorphic either to the group $PSL_2(\mathbb R)$ or to the product $PSL_2(\mathbb R) \times PSL_2(\mathbb R)$. But these groups are excluded in Lemma 19.5  and in Theorem 19.7 in \cite{loops}, pp. 252-253. 

If $\hbox{dim}\ \mathcal R \ge 1$ then there exists a normal subgroup $\mathcal K$ of $Mult(L)$ which is abelian. 
According to Lemma \ref{bruck}  the orbit $\mathcal K(e)$ is a normal subloop of $L$. 
For $\mathcal K(e)=e$ the inner mapping group $Inn(L)$ contains the non-trivial normal subgroup $\mathcal K$ of $Mult(L)$ which is a contradiction. 

If the orbit $\mathcal K(e)$ is the whole loop $L$ then $\hbox{dim} \mathcal K=2$ otherwise the group $Mult(L)$ does not act effectively on $L$. Moreover, the group $\mathcal K$ operates sharply transitively on $L$.  
Hence we have $Mult(L)={\mathcal K} \rtimes Inn(L)$. If the group $Mult(L)$ is non-solvable, then it is  the semidirect product of the $2$-dimensional abelian group $\mathcal K$ by a Lie group $S$ locally isomorphic either to $GL_2(\mathbb R)$ (cf. Subcase III.8 and IV.2, p. 635) or to 
$SL_2(\mathbb R)$ (cf. Subcase III.1, p. 633, and Subcase IV.1, p. 635). In all these cases the subgroup locally isomorphic to $SL_2(\mathbb R)$ acts irreducibly on $\mathcal K$ which is the unique $2$-dimensional sharply transitive normal subgroup of the group $Mult(L)$. Hence we can identify the elements of $L$ with the elements of $\mathcal K$. But then the group $\mathcal K$ is contained in the group topologically generated by the left translations of $L$ and this gives a contradiction to Corollary 17.8 in \cite{loops}, p. 218. 

Hence the orbit $\mathcal K(e)$ is a $1$-dimensional normal subloop of $L$. 
As the factor loop $L/\mathcal K(e)$ is a connected $1$-dimensional loop such that the group topologically generated by all its left and right translations is a factor group of the group $Mult(L)$ we have $L/\mathcal K(e)$ is a $1$-dimensional Lie group (see Theorem 18.18 in \cite{loops}, p. 248).  Then the group $Mult(L)$ contains a normal subgroup $M$ of codimension $1$ such that ${\mathcal K} \le M$ and the group $M$ leaves every left coset $\{ x \mathcal{K}(e); \ x \in L \}$ invariant (see Lemma \ref{bruck}).  

According to \cite{mostow}, §10, if the group $Mult(L)$ is non-solvable, then it is locally isomorphic either to the direct product $PSL_2(\mathbb R) \times \mathcal{L}_2$, where 
$\mathcal{L}_2 =\{ x \mapsto a x+b; \ a>0, b \in \mathbb R \}$ (cf. Subcase II.10, p. 632),  or to the semidirect product of a normal group $\mathcal K$ isomorphic to $\mathbb R^n$, 
$n \ge 2$, with a Lie group $S$ locally isomorphic to $GL_2(\mathbb R)$ such that the subgroup locally isomorphic to $SL_2(\mathbb R)$ acts irreducibly on $\mathcal K$ (cf. Subcase III.8 and IV.2, p. 635). 
The groups locally isomorphic to $PSL_2(\mathbb R) \times \mathcal{L}_2$ are excluded in Lemma 
19.5 in \cite{loops} (p. 252). If the group $Mult(L)$ is locally isomorphic to 
$\mathcal K \rtimes GL_2(\mathbb R)$, then the subgroup $M$ is locally isomorphic to 
$\mathcal K \rtimes SL_2(\mathbb R)$. 
Since $\hbox{dim}\ \mathcal{K} \ge 2$ the group $\mathcal{K}$ has a subgroup $\hat{ \mathcal{K}}$ of codimension $1$ such that 
$\hat{ \mathcal{K}}(e)=e$. As $\mathcal{K}$ acts transitively on the $1$-dimensional Lie group $\mathcal{K}(e)$ the group $\hat{\mathcal{K}}$ 
fixes $\mathcal{K}(e)$ elementwise. As $\hat{\mathcal{K}} < M$ the group  
$M$ does not act effectively on $\mathcal{K}(e)$. Then there exists a normal subgroup $N$ of $M$ which fixes every element of 
$\mathcal{K}(e)$ and hence  $N \cap \mathcal K= \hat{ \mathcal{K}}$.  Since the abelian group $\mathcal{K}$ is the unique normal subgroup of $M$ we have a contradiction and the Lemma is proved.  
\end{proof}

\bigskip
\noindent
\begin{Rem} \label{solvablemult} If the group $Mult(L)$ topologically generated by all left and right translations of a  $2$-dimensional connected simply connected topological proper loop is a solvable Lie group, then $Mult(L)$ is a semidirect product of the abelian group $M \cong \mathbb R^n$, for some $n \ge 2$, by a group $S$ isomorphic to 
$\mathbb R$, such that $M=Z \times M_1$, where $Z \cong \mathbb R$ is the centre of $Mult(L)$ and $M_1 \cong \mathbb R^{n-1}$ is the stabilizer of 
$e \in L$ in $Mult(L)$ which does not contain any non-trivial normal subgroup of $Mult(L)$ (cf. \cite{loops}, Theorem 28.1, p. 338).  \end{Rem} 

\noindent
{\bf Proof of Theorem 1}
\begin{proof}  
According to Lemma \ref{affinities} the multiplication group $Mult(L)$ is solvable. As the loop $L$ is proper the group $Mult(L)$ has dimension at least $3$. We may assume that $L$ is simply connected and hence it is homeomorphic to $\mathbb R^2$. Using the classification of Mostow (cf. \cite{mostow}, §10) every solvable Lie group with dimension $\ge 3$ acting transitively on the plane $\mathbb R^2$ is locally isomorphic to one of the Lie groups in the Subcases I.3, II.1, II.3, II.5, II.7, II.11, II.12, II.13, III.3, III.4, III.5 and III.7. It follows from Remark \ref{solvablemult} that the commutator subgroup of $Mult(L)$ is abelian and $Mult(L)$ 
has a $1$-dimensional centre.    
Hence a direct computation shows that the group $Mult(L)$ is isomorphic either to the  $3$-dimensional non-commutative nilpotent Lie group or to the direct product of a $1$-dimensional Lie group with the $2$-dimensional non-commutative solvable Lie group or its Lie algebra has the form: 
\begin{equation} {\bf mult_1(L)}=\langle \frac{\partial }{\partial x}, w(x) \frac{\partial }{\partial y}, \cdots , w^{(n)}(x) \frac{\partial }{\partial y} \rangle , n> 1, \hbox{(cf. Subcase II.3, p. 628)}. \nonumber  \end{equation} 
\noindent
The $3$-dimensional groups are excluded in \cite{loops} 
by Theorem 23.12 (p. 308) and by Theorem 23.7 (p. 299). For every $n>1$ the abelian normal subgroup $M$ of 
codimension $1$ of the Lie group  $Mult_1(L)$ 
belonging to the Lie algebra  ${\bf mult_1(L)}$ is generated by the set $\{ w(x) \frac{\partial }{\partial y}, w^{(1)}(x) \frac{\partial }{\partial y}, \cdots , w^{(n)}(x) \frac{\partial }{\partial y} \}$, $n> 1$.   
The group $M$ has $1$-dimensional centre of the group  $Mult_1(L)$ precisely if the Lie algebra of $M$ contains the generator $\frac{\partial }{\partial y}$ (cf. \cite{mostow}, Lemma 1, p. 628). Hence the function $w(x)$ has the shape $x^n$, $n >1$. Then for every $n>1$ 
the group  $Mult_1(L)$ is  nilpotent with maximal nilindex. Putting $e_1= \frac{\partial }{\partial x}$, $e_2= x^n \frac{\partial }{\partial y}$, $e_3= x^{n-1} \frac{\partial }{\partial y}$, $\cdots $, $e_{n+1}= x \frac{\partial }{\partial y}$, $e_{n+2}= \frac{\partial }{\partial y}$ we obtain the first assertion. 

As $M=Z \times Inn(L)$ (see Remark \ref{solvablemult}), where the centre of $Mult_1(L)$ 
belongs to the Lie algebra $\langle e_{n+2} \rangle $ the Lie algebra  of the group $Inn(L)$ is given by 
\begin{equation} {\bf inn(L)}= \langle e_2+a_1 e_{n+2}, e_3+a_2 e_{n+2}, \cdots , e_{n+1}+a_n e_{n+2} \rangle , \ \ a_i \in \mathbb R, i=1, \cdots , n.  \nonumber \end{equation} 
Using the automorphism $\varphi $ of the Lie algebra ${\bf mult_1(L)}=\langle e_1, e_2, \cdots , e_{n+2} \rangle $ defined by 
{\small \begin{eqnarray} 
\varphi (e_1) & = & e_1, \nonumber \\
\varphi (e_2) & = & e_2-n a_n e_3 -n a_{n-1} e_4- n a_{n-2} e_5- \cdots -n a_2 e_{n+1}- a_1 e_{n+2}, 
\nonumber \\ 
\varphi (e_3) & = & e_3- (n-1) a_n e_4 -(n-1) a_{n-1} e_5- \cdots -(n-1) a_3 e_{n+1} - a_2 e_{n+2}, \nonumber \\ 
\vdots \nonumber \\ 
\varphi (e_n) & = & e_n-2 a_n e_{n+1}- a_{n-1} e_{n+2}, \nonumber \\ 
\varphi (e_{n+1}) & = & e_{n+1}-a_n e_{n+2}, \nonumber \\
\varphi (e_{n+2}) & = & e_{n+2} \nonumber \end{eqnarray} }
the last assertion follows. 

Since the group $G$ topologically generated by the left translations of $L$ acts  also transitively on the plane $\mathbb R^2$ and it is an at least $3$-dimensional subgroup of an elementary filiform  Lie group $Mult(L)$ 
the classification of Mostow (cf. \cite{mostow}, §10) gives the second assertion. 
\end{proof}

\noindent
{\bf Proof of Theorem 2}
\begin{proof}   
As  the  Lie algebra ${\bf g}=\langle e_1, e_2, \cdots , e_{n+2} \rangle $, $n \ge 1$, of the elementary filiform Lie group $G=\mathbb F _{n+2}$ 
is isomorphic to the Lie algebra of matrices 
{\small \begin{equation} \left \{ \begin{pmatrix} 
0 & a_1 & a_2 & \dots & a_{n-1} & a_n & b \\
0 & 0 & 0 &  \dots & 0 & 0 & -c \\
0 & -2c & 0 & \dots & 0 & 0 & 0 \\
0 & 0  & -3c & \dots & 0& 0 & 0 \\
\hdotsfor{7} \\
0 & 0  & 0 & \dots & -nc & 0 & 0 \\
0 & 0  & 0 & \dots & 0 & 0 & 0 \end{pmatrix}; \ a_i,b,c \in \mathbb R, i=1,2,\cdots , n  \right \} \nonumber \end{equation} }
\noindent
we can represent the Lie group $G$ with the group 
of matrices 
{\small \begin{equation} G=\{ g(c,a_1,a_2, \cdots , a_{n-1}, a_n, b)= \nonumber \end{equation}
\begin{equation} \begin{pmatrix} 
1 & a_1 & a_2  & \dots & a_{n-1} & a_n & b \\
0 & 1 & 0  & 0 & \dots  & 0 & -c \\
0 & -2c & 1 & 0 & \dots  & 0 & c^2 \\
0 & 3c^2  & -3c & 1 & \dots & 0 & -c^3 \\
\hdotsfor{7} \\
0 & (-1)^{n-1} \binom{n}{1} c^{n-1}  & (-1)^{n-2} \binom{n}{2} c^{n-2} & \dots  & (-1)  \binom{n}{n-1} c^{1}  & 1 & (-1)^n c^n \\
0 & 0  & 0 & \dots  &  0 & 0 & 1 \end{pmatrix}; \nonumber \end{equation} 
\begin{equation} a_i, b, c \in \mathbb R, i=1,2, \cdots , n \}. \end{equation} } 
\noindent
Since all elements of $G$ have a unique decomposition 
\[ g(u,0, \cdots ,0,z) g(0,v_1,v_2, \cdots ,v_n,0) \] 
the continuous functions 
$v_1(u,z)$, $v_2(u,z)$, $\cdots $, $v_n(u,z)$ determine a continuous section $\sigma :G/H \to G$ with $H=\{ g(0,v_1,v_2, \cdots ,v_n,0);$  $ v_i \in \mathbb R, i=1, \cdots ,n \}$ 
given by 
\begin{eqnarray}
g(u,0, \cdots ,0,z)H & \mapsto & g(u,0, \cdots ,0,z) g(0,v_1(u,z),v_2(u,z), \cdots ,v_n(u,z),0) \nonumber \\
& = & g(u,v_1(u,z),v_2(u,z), \cdots ,v_n(u,z),z). \nonumber \end{eqnarray}
\noindent
The image $\sigma (G/H)$ acts sharply transitively on the factor space $G/H$ if and only if for every 
$(u_1,z_1), (u_2,z_2) \in \mathbb R^2$ the equation 
\begin{equation} g(u,v_1(u,z),v_2(u,z), \cdots ,v_n(u,z),z) g(u_1,0, \cdots ,0,z_1)= \nonumber \end{equation} 
\begin{equation} g(u_2,0, \cdots ,0,z_2) g(0,t_1,t_2, \cdots ,t_n,0)  \end{equation}
has a unique solution $(u,z) \in \mathbb R^2$ with a suitable element $g(0,t_1,t_2, \cdots ,t_n,0) \in H$. From $(4)$ we obtain the equations  
{\small \begin{eqnarray} 
u & = & u_2-u_1 \nonumber \\ 
t_1 &=& v_1(u,z)-2 u_1 v_2(u,z)+ \cdots +(-1)^{n-1} \binom{n}{1} u_1^{n-1}v_n(u,z) \nonumber \\
t_2 &=& v_2(u,z)-3 u_1 v_3(u,z)+ \cdots +(-1)^{n-2} \binom{n}{2} u_1^{n-2}v_n(u,z) \nonumber \\
\vdots \nonumber \\ 
t_n &=& v_n(u,z) \nonumber \\
0 &=& z+z_1-z_2-u_1 v_1(u,z)+u_1^2 v_2(u,z)+ \cdots +(-1)^n u_1^n v_n(u,z). \nonumber \end{eqnarray} }  
\noindent 
Hence the equation $(4)$ has a unique solution if and only if 
for every $u_0=u_2-u_1$ and $u_1 \in \mathbb R$ the function 
$f:z \mapsto z-u_1 v_1(u_0,z)+u_1^2 v_2(u_0,z)+ \cdots +(-1)^n u_1^n v_n(u_0,z): \mathbb R \to \mathbb R$ is a bijective mapping. 
Let be 
$\psi _1 < \psi _2 \in \mathbb R$ then $f(\psi _1) \neq f(\psi _2)$, e.g. 
$f(\psi _1) < f(\psi _2)$. 
We consider the inequality 
{\small \begin{equation}  0 < f(\psi _2)-f(\psi _1) =  \nonumber  \end{equation}  
\begin{equation} \psi _2 - \psi _1 - u_1[ v_1(u_0, \psi _2)- v_1(u_0, \psi _1)]+u_1^2 [v_2(u_0, \psi _2)- v_2(u_0, \psi _1)]+ \cdots + \nonumber \end{equation} 
\begin{equation} (-1)^n u_1^n [v_n(u_0, \psi _2)- v_n(u_0, \psi _1)]. \end{equation} } 
If for every $i=1, \cdots ,n$ the function $v_i(u,z)$ does not depend on the variable $z$ then $f$ is monoton and the continuous functions $v_1(u), v_2(u), \cdots ,v_n(u)$ determine a $2$-dimensional topological loop $L$. 
Now we represent the loop multiplication in the coordinate system $(u,z) \mapsto g(u, 0, \cdots ,0, z)H$: The multiplication $(u_1,z_1) \ast (u_2, z_2)$ of $L$ will be determined if we apply the left multiplication map $\sigma (g(u_1,0, \cdots ,0,z_1)H)=g(u_1,v_1(u_1),v_2(u_1), \cdots ,v_n(u_1),z_1)$ to the left coset 
$g(u_2,0, \cdots ,0,z_2)H$ and in the coset of the image element we have to find the element which lies in the set $\{ g(u,0, \cdots ,0,z)H; \ u,z \in \mathbb R\}$. We obtain 
{\small \begin{equation} (u_1,z_1) \ast (u_2,z_2)= 
(u_1+u_2, z_1+z_2-u_2 v_1(u_1)+u_2^2 v_2(u_1)+ \cdots +(-1)^n u_2^n v_n(u_1)).  \nonumber 
\end{equation} } 
\newline
\noindent
The loop $L$ is proper precisely if the set 
$\sigma (G/H)=\{ g(u,v_1(u),v_2(u), \cdots, v_n(u),z); \\ u,z \in \mathbb R \}$ 
generates the whole group $G$.    
The set $\sigma (G/H)$ contains the centre of $G$  
\[ Z=\{ g(0,0, \cdots ,0,z); \ z \in \mathbb R \} \]
and the subset 
\[ F=\{ g(u,v_1(u),v_2(u), \cdots ,v_n(u),0); \ u \in \mathbb R \}. \]
The computation of the product  
\[ g(u_1,v_1(u_1),v_2(u_1), \cdots ,v_n(u_1),0) g(u_2,v_1(u_2),v_2(u_2), \cdots ,v_n(u_2),0)  \]  
yields that the group $Z$ and the subgroup $<F>$ topologically generated by the set $F$ generates $G$ if and only if the $(n+1)$-th component $v_n(u_1)+v_n(u_2)$ of the elements of $<F>$ are not a homomorphic image of the first component $u_1+u_2$. This is the case precisely if the function $v_n$ is non-linear and the first assertion follows. 

According to Proposition 18.16 in \cite{loops}, p. 246, the filiform Lie group $G$  coincides with the group topologically generated by all translations of the loop $L$  given by the multiplication (1) 
if and only if for every $y \in L$ the map $f(y): x \mapsto y \lambda _x \lambda _y^{-1}:L \to L$ is an element of $H=\{ g(0,t_1, \cdots ,t_n,0); t_i \in \mathbb R, i=1, \cdots ,n \} $. This is equivalent to the 
condition that the mapping 
{\small \begin{eqnarray} g(x,0, \cdots ,0,y) H & \mapsto & [g(u,v_1(u),v_2(u), \cdots ,v_n(u),z)]^{-1} \nonumber \\
& & [g(x,v_1(x),v_2(x), \cdots ,v_n(x),y)] g(u,0, \cdots ,0,z) H \nonumber \end{eqnarray} }
\noindent
has the form 
{\small \begin{equation} g(x,0, \cdots ,0,y) H \mapsto g(0,s_1(u,z),s_2(u,z), \cdots ,s_n(u,z),0) g(x,0, \cdots ,0,y) H \nonumber \end{equation} }
for suitable functions $s_1(u,z),s_2(u,z), \cdots ,s_n(u,z)$. This gives the relation 
\begin{equation} g(x,v_1(x), \cdots ,v_n(x),y)g(u,0, \cdots ,0,z) H= \nonumber \end{equation} 
\begin{equation} g(u,v_1(u), \cdots ,v_n(u),z)g(0,s_1(u,z), \cdots ,s_n(u,z),0) 
 g(x,0, \cdots ,0,y) H \nonumber \end{equation} 
or the equation 
{\small \begin{equation}
g(x,v_1(x), \cdots ,v_n(x),y)g(u,t_1, \cdots ,t_n,z) =   \nonumber  \end{equation} 
\begin{equation} g(u,v_1(u), \cdots ,v_n(u),z)g(0,s_1(u,z), \cdots ,s_n(u,z),0) g(x,0, \cdots ,0,y)  \end{equation} } 
for a suitable $g(0,t_1, \cdots ,t_n,0) \in H$. 
The equation (6) yields the equations 
\noindent
{\small \begin{eqnarray} t_1 &=& s_1(u,z)+ \cdots + (-1)^{n-1} \binom{n}{1} x^{n-1} s_n(u,z)+ \nonumber \\
& & v_1(u)-v_1(x)+ \cdots +(-1)^{n-1} \binom{n}{1} x^{n-1}v_n(u)- (-1)^{n-1} \binom{n}{1} u^{n-1}v_n(x)  \nonumber  \\
t_2 &=& s_2(u,z)+ \cdots + (-1)^{n-2} \binom{n}{2} x^{n-2} s_n(u,z)+  \nonumber \\
& & v_2(u)-v_2(x)+ \cdots +(-1)^{n-2} \binom{n}{2} x^{n-2} v_n(u)-(-1)^{n-2} \binom{n}{2} u^{n-2} v_n(x) \nonumber \\
\vdots \nonumber \\ 
t_n &=& s_n(u,z)+v_n(u)-v_n(x) \nonumber \end{eqnarray} }
and 
{\small \begin{equation}  -x s_1(u,z)+ \cdots + (-1)^{n} x^{n} s_n(u,z)- v_1(u) x+ \cdots + (-1)^{n} x^{n} v_n(u) =  \nonumber \end{equation}
\begin{equation} -u v_1(x)+ \cdots +(-1)^n u^n v_n(x). \end{equation} }
\noindent
Since the right hand side of the equation (7) does not depend on the variable $z$ we have $s_i(u,z)=s_i(u)$ 
for all $i=1, \cdots ,n$. 
Hence the equation (6) is satisfied precisely if there are continuous functions $s_i:\mathbb R \to \mathbb R$, $i=1, \cdots ,n$, depending on the variable $u$ such that the equation (7) holds. The 
loop $L$ is commutative precisely if $s_i(u)=0$ for all $i=1, \cdots ,n$  (cf. Lemma 18.16 in \cite{loops}, p. 246). 
If the functions $v_i:\mathbb R \to \mathbb R$, $i=1,2,\cdots ,n$ are polynomials, then the comparison of the coefficients yields the last assertion. 
\end{proof}

\begin{Lemma} \label{subalgebra} Let $V$ be a non-commutative subalgebra of the elementary filiform Lie algebra 
$\mathbb F _{n+2}$, $n \ge 1$. Then $V=V_i$ has a basis $\{ e_1+t_1, e_i, e_{i+1}, \cdots ,e_{n+2} \}$ with a fixed $i \in \{ 2, \cdots ,n+1 \}$ and $t_1 \in \langle e_2, e_3, \cdots , e_{i-1} \rangle $. 
\end{Lemma} 
\begin{proof} If $V$ is not commutative, then  $V$ is a filiform Lie algebra of dimension $n+4-i$ 
with $2 \le i \le n+1$ and has a basis of the shape $\{ e_1+t_1, e_i+t_i, e_{i+1}+ t_{i+1}, 
\cdots , e_{n+1}+t_{n+1}, e_{n+2} \}$ with $t_1 \in \langle e_2, e_3, \cdots ,e_{n+2} \rangle $ and $t_{i+j} \in \langle e_{i+j+1}, \cdots ,e_{n+2} \rangle $, $0 \le j <n+1-i$. A successive addition of a suitable linear combination $\sum \limits _{k=0}^j \lambda _{n+2-k} e_{n+2-k}$ 
to $t_{n+1-j} $ shows that $V$ contains the elements $e_{n+2}, e_{n+1}, \cdots ,e_i$ and the 
assertion follows. \end{proof}

\noindent
{\bf Proof of Theorem 3}
\begin{proof} For $n=1$ the equation (5) in the proof of Theorem 2 is linear in the variable $u_1 \in \mathbb R$. Hence the function $f$ is monoton if and only if $v_1(u,z)=v_1(u)$ and the first assertion follows from Theorem 2. 

Let $G$ be the group of matrices (3) which is isomorphic to $\mathbb F_{n+2}$ (cf. Proof of Theorem 2) and let $H$ be the subgroup $\{ g(0,t_1,t_2,\cdots ,t_n,0); \ t_i \in \mathbb R, i=1,2, \cdots ,n \}$.   
The set $\Lambda _{v_1}=\{ \lambda _{(u,v)}; \ (u,v) \in L_{v_1} \}$ 
of all left translations of the loop $L_{v_1}$ defined by (2) in the group $G$ has the shape  
\[ \Lambda _{v_1}=\{ g(u,v_1(u),0,0, \cdots , 0, -\frac{v_1(u) u}{2}+z); \ u,z \in \mathbb R \}.  \]  
An arbitrary transversal $T$ of the group $H$ in the group $G$ has the form 
\[ T=\{ g(x,h_1(x,y), \cdots ,h_n(x,y), y);\  x,y \in \mathbb R \}, \]
where $h_j(x,y): \mathbb R^2 \to \mathbb R$, $j=1,2,\cdots ,n$ are continuous  functions with $h_j(0,0)=0$. 
According to Lemma \ref{kepka} the group $G$ is isomorphic to the multiplication group $Mult(L)$ of the loop $L_{v_1}$ if and only if 
$[T, \Lambda _{v_1}] \subset H$ and the set $\{ \Lambda _{v_1}, T \}$ generates the group $G$.  
For all $x,y,u,z \in \mathbb R$ the products $a^{-1} b^{-1} a b$ for every $a \in T$ and 
$b \in \Lambda _{v_1}$ are contained in $H$ if and only if the equation 
{\small \begin{equation} x v_1(u)=(-1)^{n+1} u^n h_n(x,y)+(-1)^{n} u^{n-1} h_{n-1}(x,y)+ \cdots +(-1)^2 u h_1(x,y)  \end{equation} }
\noindent
holds for all $x,y,u \in \mathbb R$. 
If $x=0$ then the equation (8) reduces to 
{\small \begin{equation} 0=(-1)^{n+1} u^n h_n(0,y)+(-1)^{n} u^{n-1} h_{n-1}(0,y)+ \cdots +(-1)^2 u h_1(0,y). \end{equation} } 
\noindent
Since the polynomials $u, u^2, \cdots , u^n$ are linearly independent the equation (9) is satisfied if and only if 
$h_n(0,y)=h_{n-1}(0,y)= \cdots =h_1(0,y)=0.$ 
As the function $v_1: \mathbb R \to \mathbb R$ depends only on the variable $u$ and the functions $h_i(x,y):\mathbb R^2 \to \mathbb R$, $i=1,2, \cdots , n$ are independent from the variable $u$, the equation (8) holds precisely if $h_1(x,y)=a_1 x$, $h_2(x,y)=a_2 x$, $\cdots $, $h_n(x,y)=a_n x$, where $a_1,a_2, \cdots , a_n \in \mathbb R$. According to Lemma \ref{subalgebra} the set $\{ \Lambda _{v_1}, T \}$ generates the group $G$ if and only if 
$a_n$ is different from $0$ since then the Lie algebra of the non-commutative group generated by the set $\{ \Lambda _{v_1}, T \}$ contains elements of the shape $e_2+s$ with $s \in \langle e_3, e_4, \cdots , e_{n+2} \rangle $.  
\end{proof}

Author's address: Institute of Mathematics, University of Debrecen, \\
P.O.Box 12, H-4010 Debrecen, Hungary, figula@math.klte.hu

\end{document}